\documentclass{amsart}

\usepackage{enumitem, hyperref}

\usepackage{fouridx}
\renewcommand{\star}[1]{\ensuremath{{\fourIdx{*}{}{}{}{#1}}}}
\newcommand{\std}[1]{{\fourIdx{\circ}{}{}{}{#1}}}

\newcommand{\C}{\mathbb{C}}
\newcommand{\SC}{\star\C}

\newtheorem{Thm}{Theorem}
\newtheorem*{Thm*}{Theorem}
\newtheorem{Lemma}{Lemma}

\title[Roots depend continuously]{Yet another proof that the roots of a polynomial depend continuously on the coefficients}

\author{David A. Ross}
\address{Department of Mathematics\\University of Hawaii at Manoa\\
Honolulu, HI 96822}
\email{ross@math.hawaii.edu}



\date{6 July 2022}

\begin{document}

\begin{abstract}The roots of a complex polynomial depend continuously on the coefficients; that is, an infinitesimal perturbation of the coefficients results in an infinitesimal perturbation of the roots.  A short, straightforward proof of this is possible using infinitesimals.
\end{abstract}

\maketitle

The coefficients of a polynomial with complex coefficients depend continuously on its roots; that is an immediate consequence of the elementary \emph{Vi\'{e}te Formulas}, known in some form as long ago as the 16th century.

The inverse, that the roots depend continuously on the coefficients, remains true, but is surprisingly tricky to prove---even to formulate correctly!---since, for example, a slight perturbation of the coefficients can cause roots to coalesce.  Most of the proofs in the standard literature employ relatively heavy machinery (see, for example,
\cite{cuck-corb89, harr-mart87, henr-isbe51, hiro20, mard66, ostr, rahm-schm02, whit72}).

Recently the author, in collaboration with Mel Nathanson, produced an entirely elementary proof \cite{NR} of this result.  While that paper uses nothing beyond high school algebra, the current author's contribution was motivated by an argument using infinitesimals in the sense of Abraham Robinson's \emph{nonstandard analysis} \cite{Robinson}.  This paper is an exposition of the infinitesimal approach.  While the proof here carries less information than the standard proof,
the lack of necessity for careful estimates make the proof shorter and arguably more transparent.  (See Section~\ref{remarks} for a further discussion.)

To make these arguments work requires just a few ideas from this nonstandard machinery;  our use of infinitesimals is not much different from what Euler might have used three centuries ago.  Readers unfamiliar with this machinery should read
Appendix~\ref{section:nonstandard} now.

\section*{Acknowledgement.}  This paper wouldn't exist in the absence of many discussions of the basic problem with Mel Nathanson, who introduced it at CANT 2022 (24-27 May 2022).

\section{Results}

Let \[f(z)=a_0+a_1 z+\dots+a_nz^n\]be a polynomial with complex coefficients and $a_n\neq 0$. By the Fundamental Theorem of Algebra, $f(z)$ fully factors as \[f(z)=a_n\prod\limits_{1\le i\le n}(z-r_i)\]
where $r_1, r_2,\dots, r_n$ are the roots of $f(z)$, possibly with duplication.

Of course it suffices that any complex polynomial $f(z)$ has a single root in $\C$; the full factorization is then an induction on degree.  For that reason, the fact that the extension $\SC$ is algebraically closed shows that \emph{nonstandard} polynomials, with coefficients in $\SC$, can likewise be fully factored.

In particular, suppose
\[g(z)=b_0+b_1 z+\dots+b_nz^n=b_n\prod\limits_{1\le i\le n}(z-s_i)\]
is a nonstandard polynomial, so each $b_i$ is in $\SC$.  Say that $g(z)$ is an \emph{infinitesimal deformation} of $f(z)$ provided $a_i\approx b_i$ for each $i\in\{0,1,\dots,n\}$.   Say that $g(z)$ is \emph{infinitesimally aligned} with $f(z)$ provided  there are orderings
$r_1,r_2,\dots,r_n$ (respectively, ${s}_1,{s}_2,\dots,{s}_n$) of the (not necessarily distinct) roots of $f(z)$ (respectively, $g(z)$), where ${s}_i\approx r_i$ for $i\in\{1,2,\dots,n\}$.

The main result of this paper is the following.

\begin{Thm}\label{theorem:main}Suppose
\[g(z)=b_0+b_1 z+\dots+b_nz^n=b_n\prod\limits_{1\le i\le n}(z-s_i)\]
is an infinitesimal deformation of
\[f(z)=a_0+a_1 z+\dots+a_nz^n=a_n\prod\limits_{1\le i\le n}(z-r_i)\]
(where $a_n\neq 0$).  Then $g(z)$ is infinitesimally aligned with $f(z)$.
\end{Thm}

\section{Two lemmas}

Before proving Theorem~\ref{theorem:main} we need two lemmas.

\begin{Lemma}\label{lemma:stdofpolynomial}Suppose $f(z)=\sum_{i=0}^na_iz^i$, where $a_n\neq 0$, is a polynomial with coefficients in $\C$, and $g(z)=\sum_{i=0}^nb_iz^i$ is a polynomial with coefficients in $\SC$.  Consider the statements:
\begin{enumerate}[label=(\roman*)]
\item $g(z)$ is infinitesimally aligned with $f(z)$ and $b_n\approx a_n$
\item $g(z)\approx f(z)$ for all $z\in\C$
\item $g(z)$ is an infinitesimal deformation of $f(z)$
\end{enumerate}
Then $(i)\Rightarrow(ii)\Leftrightarrow(iii)$
\end{Lemma}

\begin{proof} Assume $(i)$ holds.  For any (standard)$z\in\C$:
\[\std{g(z)}=\std{\displaystyle\left(b_n\prod\limits_{1\le i\le n}(z-s_i)\right)}=\std{b_n}\prod\limits_{1\le i\le n}({z}-\std{s_i})=a_n\prod\limits_{1\le i\le n}({z}-r_i)=f({z})\]
This proves $(i)\Rightarrow(ii)$.

Now, suppose $(ii)$ holds.  Suppose first that we know that all the coefficients of $g(z)$ are finite.  For any standard (standard)$z\in\C$,
\[ \sum_{i=0}^n\std{b_i}z^i=\std{\displaystyle\left(\sum_{i=0}^n{b_i}z^i\right)}=\std{g(z)}=f(z)=\sum_{i=0}^na_iz^i\]
and two polynomials agree on $\C$ only if they have the same coefficients; that is, $b_i\approx a_i$ for every $i$, proving $(ii)\Rightarrow(iii)$.  It remains to show that if $g(z)$ takes only finite values for standard $z$, then it has finite coefficients.  If not, then there is a coefficient $j\le n$ with $|b_j|$ largest.  But then for any (standard)$z\in\C$,
\[0=\std{\displaystyle\left(\frac{g(z)}{b_j}\right)}=\std{\displaystyle\left(\sum_{i=0}^n\frac{b_i}{b_j}z^i\right)}=\sum_{i=0}^n\std{\displaystyle\left(\frac{b_i}{b_j}\right)}z^i\]
so $\std{(b_i/b_j)}$ must be $0$ for all $i$; but$\std{(b_j/b_j)}=1$, a contradiction.

Finally, suppose that $(iii)$ holds, and $z\in\C$ is standard.  Then
\[ \std{g(z)} =\std{\displaystyle\left(\sum_{i=0}^n{b_i}z^i\right)} =\sum_{i=0}^n\std{b_i}z^i = \sum_{i=0}^na_iz^i={f}(z)\]
proving $(iii)\Rightarrow(ii)$
\end{proof}

\subsection*{Remarks} (1)~In the proof of $(ii)\Rightarrow(iii)$, it is only required that $g(z)\approx f(z)$ for infinitely many values (or even at least $n$ different values) of $z$.  (2)~The implication $(i)\Rightarrow(iii)$ in this lemma is also a consequence of the Vi\'{e}te formulas, which give the coefficients of the polynomial as elementary symmetric functions of the roots. (3)~The requirement that $b_n\approx a_n$ in $(i)$ is necessary to account for the fact that the roots of a polynomial only determine the coefficients up to a constant multiple.

The second lemma is the key to the full result.

\begin{Lemma}\label{lemma:existsanearbyroot}Suppose $g(z)=\sum_{i=0}^nb_iz^i$ is an infinitesimal deformation of $f(z)=\sum_{i=0}^na_iz^i$, where $a_n\neq 0$.  Then for every root ${s}$ of $g(z)$, $r=\std{s}$ is a root of $f(z)$.
\end{Lemma}

\begin{proof}First, we note that ${s}$ is not infinite; otherwise,
\[0=\frac{g({s})}{{s}^n}=b_n+\frac{b_{n-1}}{{s}}+\dots+\frac{b_0}{{s}^n}\approx a_n+0+\dots+0=a_n\]
since each $b_i\approx a_i$ is finite.  But $a_n\neq 0$ by hypothesis, a contradiction.

Since ${s}$ is finite, it has a standard part $\std{{s}}$, and
\[ {f}(\std{{s}}) = \sum_{i=0}^na_i(\std{{s}})^i =\sum_{i=0}^n\std{b_i}(\std{{s}})^i =  \std{\displaystyle\left(\sum_{i=0}^n{b_i}{{s}}^i\right)}=\std{0}=0   \]
So $r=\std{{s}}$ is a root of $f(z)$.
\end{proof}

\section{Proof that the roots depend continuously on the coefficients.}

We can now prove Theorem~\ref{theorem:main}. The proof is by induction on $n$. We are given that
\[g(z)=b_0+b_1 z+\dots+b_nz^n=b_n\prod\limits_{1\le i\le n}(z-s_i)\]
is an infinitesimal deformation of
\[f(z)=a_0+a_1 z+\dots+a_nz^n=a_n\prod\limits_{1\le i\le n}(z-r_i)\]

When $n=1$ the only root of $f(z)$ is $r_1=-a_0/a_1$, and the only root of $g(z)$ is ${s}_1=-b_0/b_1$. Moreover, $\std{(-b_0/b_1)}=-\std{b_0}/\std{b_1}=-a_0/a_1$ by the usual rules of nonstandard arithmetic on $\SC$.  It follows that ${s}_1\approx r_1$.

Now, suppose the result is true for polynomials of degree $n-1$.  Let ${s}$ be a root of $g(z)$.  Since $g(z)$ is an infinitesimal deformation of $f(z)$, by Lemma~\ref{lemma:existsanearbyroot} there is a root $r$ of $f(z)$ with $r\approx{s}$.  Write
$f(z)={(z-r)}\hat{f}(z)$ and $g(z)={(z-s)}\hat{g}(z)$, where $\hat{f}$ and $\hat{g}$ are polynomials of degree $n-1$ over their respective fields.  Then for any $z\in\C$:
\begin{equation}\label{equation:hateqn}(z-r)(\hat{f}(z)-\hat{g}(z))=f(z)-g(z)-\hat{g}(z)(s-r)\end{equation}
When $z\neq r$, $\hat{g}(z)=g(z)/(z-s)$ is finite, so the right hand side of (\ref{equation:hateqn}) is infinitesimal, so
$\hat{f}(z)-\hat{g}(z)\approx 0$, or $\hat{f}(z)\approx\hat{g}(z)$.  By Lemma~\ref{lemma:stdofpolynomial} (and the remarks following) $\hat{g}(z)$ is an infinitesimal deformation of $\hat{f}(z)$.

By the induction hypothesis
there are orderings $r_1,r_2,\dots,r_{n-1}$ (respectively, ${s}_1,{s}_2,\dots,{s}_{n-1}$) of the roots of $\hat{f}(z)$ (respectively, $\hat{g}(z)$) with ${s}_i\approx r_i$ for every $i\in\{1,\dots,n-1\}$.  Put $r_n=r$ and  ${s}_n={s}$, then ${s}_i\approx r_i$ for every $i\in\{1,\dots,n\}$,
proving that $g(z)$ is infinitesimally aligned with  $f(z)$.

The theorem is proved.

\section{Remarks}\label{remarks}

Lemma~\ref{lemma:existsanearbyroot} in this paper is the analogue of \cite[Theorem 4]{NR}, and plays a similar role in the later induction.  The difference is subtle: in \cite[Theorem 4]{NR}, given a root of $f(z)$, we show that small perturbations $g(z)$ have nearby roots.  Lemma~\ref{lemma:existsanearbyroot}, on the other hand, shows that given a root of a perturbation $g(z)$ of $f(z)$, there is a nearby root of $f(z)$.  The lemma here gives a possibly more direct route to the full theorem, but the proof in \cite{NR} has an additional benefit: given a tolerance $\epsilon$, it gives a constructive bound $\delta=\delta(\epsilon)$ so that if the perturbation is kept smaller than $\delta$ then its roots are within the tolerance $\epsilon$ of the roots of $f$.  I do not see any way to extract this bound from the proof in this paper.

Moreover, while the proofs in \cite{NR} apply to any algebraically closed field with an absolute value, the proofs here (especially that of Lemma~\ref{lemma:existsanearbyroot}) use the standard part map in an essential way; this requires the absolute value on the underlying field to be complete and Archimedean, so by Ostrowski's Theorem we must be working on $\C$.

In the proof of Theorem~\ref{theorem:main}, it was not necessary to give an explicit form for $\hat{f}(z)$ and $\hat{g}(z)$ (as we did in the corresponding section of \cite{NR}).  Again, while this simplification is a byproduct of the infinitesimal approach, it loses information about the nature of the approximation that is present in the standard paper.

Finally, note that the wording of Theorem~\ref{theorem:main}, which references nonstandard elements, is equivalent to the following standard statement:

\begin{Thm*}Suppose
\[f(z)=a_0+a_1 z+\dots+a_nz^n=a_n\prod\limits_{1\le i\le n}(z-r_i)\]
is a polynomial over $\C$, where $a_n\neq 0$.  For every $\epsilon>0$ there exists a $\delta>0$ such that
whenever
\[g(z)=b_0+b_1 z+\dots+b_nz^n=b_n\prod\limits_{1\le i\le n}(z-s_i)\]
is polynomial with $|a_i-b_i|<\delta$ for $0\le i\le n$,
then there are orderings $\{r_i\}_{i=1}^n$ and $\{s_i\}_{i=1}^n$ of the roots of $f$ and $g$ (respectively), possibly with duplication, so that $|r_i-s_i|<\epsilon$ for all $i$.
\end{Thm*}

The proof of the equivalence is a straightforward exercise in nonstandard analysis, but does require a more sophisticated nonstandard model than the simple extension of $\C$ that needed for the earlier sections (and describe below).

\appendix

\section{All the infinitesimals we need}\label{section:nonstandard}

We assume the existence of an algebraically closed field $\star{\mathbb{C}}$ extending ${\mathbb{C}}$, with some extra properties.  In addition to the \emph{standard} elements of ${\mathbb{C}}$, this field also contains some new, \emph{nonstandard} elements.  Among those are \emph{infinitesimals}, namely numbers $z$ with the property that $|z|<1/n$ for every positive integer $n$ (where the usual norm $|\cdot|$ on $\C$ is extended to all of $\star{\mathbb{C}}$ in a natural way, taking values in a suitable ordered field extension of ${\mathbb{R}}$).  Of course, $0$ is an infinitesimal under this definition, but we assume the existence of nonzero infinitesimals as well.  Since $\star{\mathbb{C}}$ is a field, every nonzero infinitesimal has a reciprocal; reciprocals of infinitesimals are called \emph{infinite}.

Elements of $\star{\mathbb{C}}$ which are not infinite are \emph{finite}; equivalently, $z$ is finite if $|z|<n$ for some integer $n$.  Every finite element $z$ of $\star{\mathbb{C}}$ differs by an infinitesimal from a unique element of ${\mathbb{C}}$, which we denote by $\std{z}$ and call the \emph{standard part} of $z$.  This standard part function is the modern equivalent to ``neglecting higher order terms," and respects normal arithmetic operations:  if $a, b$ are finite elements of $\star{\mathbb{C}}$ then $\std{(ab)}=\std{a}\std{b}$, $\std{(a+b)}=\std{a}+\std{b}$, and if $b$ is not infinitesimal then $\std{(a/b)}=\std{a}/\std{b}$.
If $z$ is a standard element of $\mathbb{C}$ then $\std{z}=z$.

(In the language of algebra, the finite elements of $\star{\mathbb{C}}$ form a subring of $\star{\mathbb{C}}$, and the standard part map is a ring homomorphism from $\star{\mathbb{C}}$ onto $\mathbb{C}$.)

If $a-b$ is infinitesimal we write $a\approx b$.  Thus $a$ is infinitesimal if and only if $a\approx 0$; more generally, if $a$ is finite then $\std{a}\approx a$.

For more information about extensions like $\star{\mathbb{C}}$, including constructions of the field and exploration of other properties,  the reader is referred to any good introduction to the subject, for example Robinson~\cite{Robinson},  Davis~\cite{Davis}, or Tom Lindstr{\o}m's excellent introduction \cite{Cutland}.  That said, the previous paragraphs contain everything necessary to carry out the construction here.

\end{document}